\documentclass[a4paper, english, 10pt]{amsart}

\usepackage{amsthm,amssymb,url,hyperref}
\usepackage{fullpage}
\usepackage{stmaryrd}
\usepackage[all]{xy}
\usepackage{color}
\usepackage[all]{xy}
\usepackage{tikz}
\usepackage{color,tikz}

\theoremstyle{plain}
\newtheorem{theorem}{Theorem}[section]

\newtheorem{corollary}[theorem]{Corollary}

\newtheorem{lemma}[theorem]{Lemma}

\newtheorem{proposition}[theorem]{Proposition}

\theoremstyle{definition}

\newtheorem{example}[theorem]{Example}
\newtheorem{remark}[theorem]{Remark}

\newtheorem*{remark*}{Remark}

\def\!#1{#1^\s}

\def\ker#1{\mathrm{ker}(#1)}

\def\aut#1{\mathrm{Aut}(#1)}

\def\End#1{\mathrm{End}(#1)}
\def\aff#1{\mathrm{Aff}#1}

\def\lmlt{\mathrm{LMlt}}
\def\dis{\mathrm{Dis}}

\def\what{sharp \,}

\newcommand{\BlocS}[2]{\dis(#1)_{[#2]}}

\def\s{\mathfrak{s}}

\def\comment#1{{\color{red} #1}}

\def\setof#1#2{\{#1\, : \,#2\}}

\def\Z{\mathbb Z}

\def\Q{\mathcal Q}

\def\N{\mathrm{Norm}}
\def\c#1{\mathrm{con}_{#1}}

\def\cg#1{\equiv_\alpha}
\newcommand*\xbar[1]{%
   \hbox{%
     \vbox{%
       \hrule height 0.5pt 
       \kern0.5ex
       \hbox{%
         \kern-0.1em
         \ensuremath{#1}%
         \kern-0.1em
       }%
     }%
   }%
}

\usepackage{MnSymbol}

\title{Two Galois connections for left quasigroups}
\author{Marco Bonatto}

\address[M. Bonatto]{Dipartimento di matematica e informatica - UNIFE}

\email{marco.bonatto.87@gmail.com}

\begin{document}

\begin{abstract}
We investigate two Galois connection between the congruence lattice and the lattice of subgroups of the displacement group of left quasigroups. Such connections were already studied for racks and quandles. We introduce the class of left quasigroups having congruence determined by subgroups (resp. orbits) and we extend a known result for quandles.
\end{abstract}

\maketitle
\section*{Introduction}

Left quasigroups are rather combinatorial objects, nevertheless several algebraic structures arising in different ares of mathematics have an underlying left quasigroup structure. For instance racks and quandles coming from low dimensional topology \cite{J, Matveev} and the study of Hopf algebras \cite{AG} are examples of highly structured left quasigroups. 

A tool for the study of racks is the {\it left multiplication group} and its subgroups. In particular the {\it displacement group} reflects a lot of properties of racks. In particular properties as {\it abelianness} and {\it nilpotency} in the sense of commutator theory \cite{Smith, comm} are completely determined by the correspondent properties of the displacement group \cite{CP}.

In \cite{CP}, a monotone Galois connection between the lattice of congruences and the lattice of normal subgroups of the left multiplication group of racks has been defined by the pair of operators $(\dis_*,\c{*})$. If this Galois connection provides an isomorphism between the two lattices we said that a rack has the {\it congruence determined by subgroups} (shortly, CDSg). We also discovered a second Galois connection defined by the pair of mappings $(\dis^*,\mathcal{O}_*)$ and later on in \cite{LTT} we introduced the same pair of operators between the lattice of congruences of left quasigroups and the lattice of {\it admissible subgroups} (in the case of racks, admissible subgroups and normal subgroup of the left multiplication group are the same thing). We say that a left quasigroup has the {\it congruence determined by orbits} (shortly, CDOs) whenever this second Galois connection provides an isomorphism between the two lattices.

In this paper we investigate the relation between these two Galois connections in the setting of left quasigroups. In particular, left quasigroups having CDSg and CDOs are compared in Theorem \ref{CDSg vs CDOs}. Some results in this direction were already obtained for racks in \cite{Principal}. One of the main results of the paper is Theorem \ref{nilpotent CDSg_0} in which we extend \cite[Proposition 3.17]{Principal} to idempotent left quasigroups of arbitrary cardinality.

In Section \ref{preliminary} we collect all the basics about left quasigroups and the interplay between congruences and admissible subgroups. Moreover we introduce the pairs of operators $(\dis^*,\mathcal{O}_*)$ and $(\dis_*, \c{*})$ and how they relate with the {\it Cayley kernel} and the property of being {\it faithful}. Section \ref{Galois} is dedicated to the Galois connections defined by the operators introduced in the previous section and to left quasigroups having CDSg and CDOs. The last section is about nilpotent left quasigroups having CDSg (we also provide some basic notions of commutator theory and central extensions).

\section{Preliminary results}\label{preliminary}

\subsection{Left quasigroups}

A left quasigroup is a binary algebraic structure $(Q,\cdot,\backslash)$  such that the following identities hold:
$$x\cdot(x\backslash y)\approx y\approx x\backslash (x\cdot y).$$
We will denote the $\cdot$ operation just by juxtaposition in the rest of the paper. The {\it left and right multiplication mappings} of $Q$ are defined as
$$L_x:y\mapsto x y,\quad R_x:y\mapsto y x$$
for every $x\in Q$. According to the axioms above, the map $L_x$ is a permutation for every $x\in Q$ and so we can define the {\it left multiplication group} of $Q$ as $\lmlt(Q)=\langle L_x,\, x\in Q\rangle$. Note that the $\backslash$ operation is defined as $x\backslash y=L_x^{-1} (y)$ for every $x,y\in Q$, so we ofter specify just the $\cdot$ operation for left quasigroups, from which it is usually easy to get the left multiplication mappings and their inverses.

{\it Right quasigroups} can be defined analogously as binary algebraic structure $(Q,\cdot, /)$ satisfying
$$(y \cdot x)/x\approx y\approx (y/x)\cdot x.$$
The right multiplication mappings of right quasigroups are bijections.

The {\it set of idempotent elements of $Q$} is $E(Q)=\setof{x\in Q}{xx=x}$. We say that $Q$ is:

%

\begin{itemize}
\item[(i)] {\it idempotent} if $Q=E(Q)$, i.e. the identity $xx\approx x$ holds in $Q$;
\item[(ii)] a {\it rack} if the identity $x(yz)\approx (xy)(xz)$ holds (or equivalently $L_x\in \aut{Q,\cdot }$ for every $x\in Q$). Idempotent racks are called {\it quandles};
\item[(iii)] {\it latin} if the right multiplications are bijective. 
\end{itemize}

Latin left quasigroups are essentially {\it reducts} of {\it quasigroups}, namely binary algebraic structures with three binary operations $(Q,\cdot, \backslash,/)$, such that $(Q,\cdot,\backslash)$ is a left quasigroup and $(Q,\cdot, /)$ is a right quasigroup. The two type of structures have different signatures effecting congruences and subalgebras if the underlying set has infinite cardinality. In this paper we consider latin left quasigroups with signature $\{\cdot ,\backslash\}$.

A {\it term} in the language of left quasigroups $\{\cdot, \backslash\}$ is either a variable or an expression $t_1\cdot t_{2}$, $t_1\backslash t_2$, where $t_1, t_2$ are terms. A {\it Malt'sev} term is a ternary term $m$ such that $m(x,y,y)\approx x\approx m(y,y,x)$ holds. We say that a left quasigroup $Q$ is Malt'sev if the variety generated by $Q$ has a Malt'sev term. Given a left quasigroup $(Q,\cdot,\backslash)$ any $n$-ary term $t$ provides a map $t^Q:Q^n\longrightarrow Q$ called {\it term operation}. 

For further details on the universal algebraic definitions of terms and varieties of algebras see \cite{UA}.

\subsection{Congruences and admissible subgroups}
In this section we collect all the results on congruences and admissible subgroups we are using in the following. For further details see \cite{LTT, nilpotentleft}.

Let $Q$ be a set, we denote the lattice of equivalence relations on $Q$ as $Equiv(Q)$. Given $\alpha\in Equiv(Q)$ we denote the quotient set with respect to $\alpha$ by $Q/\alpha$ and the class of $x$ (with respect to $\alpha$) by $[x]_\alpha$ (we often omit the subscrit and we write just $[x]$).

Let $(Q,\cdot,\backslash)$ be a left quasigroup. We can define the following operators associating to every equivalence a
pair of subgroups of $\lmlt(Q)$:
\begin{align*}
\dis_*:\alpha\mapsto\dis_\alpha &=\langle h L_x L_y^{-1} h^{-1},\, x\,\alpha\, y,\, h\in \lmlt(Q)\rangle,\\
\dis^*:\alpha\mapsto\dis^\alpha &=\setof{ h \in \lmlt(Q)}{h(x)\,\alpha\, x\, \text{ for all } x\in Q}.
\end{align*}
The two operators $\dis_*,\dis^*$ are clearly monotone.

The first of the two groups is called the {\it displacement group relative to $\alpha$} and it is normal in $\lmlt(Q)$ by definition. In particular, we denote $\dis_{1_Q}$ just by $\dis(Q)$ and we call it the {\it displacement group} of $Q$.  According to \cite[Lemma 1.4]{LTT}, we have 
\begin{align*}
\lmlt(Q)&=\dis(Q)\langle L_x\rangle\, \text{ for every }\, x\in Q,\\
\dis(Q)&=\setof{L_{x_1}^{k_1}\ldots L_{x_n}^{k_n}}{x_1,\ldots x_n\in Q,\, n\in \mathbb{N},\, k_i\in \mathbb{Z}, \, \sum_j k_j=0}.
\end{align*} 

A congruence on $Q$ is an equivalence relation $\alpha$ such that $x y \,\alpha \, z t$ and $x \backslash y \,\alpha \, z \backslash t$ provided $x\, \alpha \, z$ and $y\, \alpha\, t$. We denote the lattice of congruences of $Q$ as $Con(Q)$. Given $\alpha\in Con(Q)$, the quotient set $Q/\alpha$ is endowed with a well-defined left quasigroup structure defined by
$$[x]_\alpha \cdot [y]_\alpha=[x\cdot y]_\alpha,\, [x]_\alpha \backslash [y]_\alpha=[x\backslash y]_\alpha$$
for every $[x]_\alpha, [y]_\alpha\in Q/\alpha$. The canonical map $x\mapsto [x]_\alpha$ provides a surjective morphism of left quasigroups from $Q$ to $Q/\alpha$. On the other hand, given a surjective morphism of left quasigroups $f:Q\longrightarrow Q'$, the relation $\ker{f}=\setof{(x,y)}{f(x)=f(y)}$ is a congruence of $Q$ and $Q'\cong Q/\ker{f}$. Note also that if $\alpha$ is a congruence of $Q$ and $[x]_\alpha\in E(Q/\alpha)$ then the block of $x$ with respect to $\alpha$ is a subalgebra of $Q$. 

 Let $N\leq \lmlt(Q)$. We can define two equivalence relation out of $N$, namely:
\begin{align*} 
 x\, \mathcal{O}_N\, y\, & \text{if and only if } x=h(y)\, \text{ for some } h\in N,\\
 x\, \c{N}\, y &\text{ if and only if } L_x L_y^{-1} \in N.
\end{align*}
The assignments above define two monotone operators $\c{*}:N\mapsto \c{N}$ and $\mathcal{O}_*:N\mapsto \mathcal{O}_N$ from the lattice of subgroups of $\lmlt(Q)$ to $Equiv(Q)$.

If $\mathcal{O}_N=1_Q$ (i.e. $N$ is transitive on $Q$), we say that $Q$ is {\it connected by $N$}. If $\lmlt(Q)$ is transitive we simply say that $Q$ is {\it connected}. If all the subalgebras of $Q$ are connected we say that $Q$ is {\it superconnected} (in particular $Q$ is connected). The class of superconnected left quasigroups has been studied in \cite{Maltsev_paper} and \cite{Super}. Connected idempotent left quasigroups and left quasigroup with a Mal'cev term are connected by their displacement group \cite[Proposition 3.6]{Maltsev_paper}.


Assume that $\alpha\in Con(Q)$, then we have a morphism of groups defined as:
\begin{equation}\label{pi_map}
\pi_\alpha:\dis(Q)\longrightarrow \dis(Q/\alpha),\quad L_{x_1}^{k_1}\ldots L_{x_n}^{k_n} \longrightarrow L_{[x_1]}^{k_1}\ldots L_{[x_n]}^{k_n}.
\end{equation}

Let us denote by $\dis(Q)_x$ the pointwise stabilizer of $x$ in $\dis(Q)$. If $\dis(Q)_x=1$ for every $x\in Q$ we say that $Q$ is {\it semiregular}. In particular, we have
\begin{align}
\dis_\alpha\leq \dis^\alpha=\ker{\pi_\alpha}\leq \dis(Q)_{[x]}=\pi_\alpha^{-1}(\dis(Q/\alpha)_{[x]})=\setof{h\in \dis(Q)}{h(x)\, \alpha\, x}
\end{align}
for every $x\in Q$. Moreover, we have the following.
\begin{proposition}\label{pi}	\cite[Proposition 3.2]{CP}
Let $(Q,\cdot,\backslash)$ be a left quasigroup and $\alpha,\beta\in Con(Q)$ such that $\alpha\leq \beta$. Then: 
\begin{itemize}
\item[(i)] $\dis_{\beta/\alpha}=\pi_{\alpha}(\dis_\beta)$. 
\item[(ii)] $\dis^{\beta/\alpha}=\pi_{\alpha}(\dis^\beta).$
\end{itemize}
\end{proposition}

We focus on the set of {\it admissible subgroups} contained in the displacement group and its interplay with congruences. The admissible subgroups we are interest in are defined in \cite{LTT} as
\begin{align*}
\N'(Q)&=\setof{N\trianglelefteq \lmlt(Q)}{N\leq \dis(Q),\, \mathcal{O}_N\leq \c{N}}\\
&=\setof{N\trianglelefteq \lmlt(Q)}{\dis_{\mathcal{O}_N}\leq N\leq \dis(Q)}.
\end{align*}
In particular $\dis_\alpha,\dis^\alpha\in \N'(Q)$ for every $\alpha\in Con(Q)$ \cite[Corollary 1.9]{LTT}. If $N\in \N(Q)$ then $\mathcal{O}_N\in Con(Q)$ and $N\leq \dis^{\mathcal{O}_N}$ \cite[Corollary 1.9]{LTT}. It is easy to check that 
$$\mathcal{O}_{\dis_\alpha}\leq \mathcal{O}_{\dis^\alpha}\leq   \alpha\leq \c{\dis_\alpha}\leq \c{\dis^\alpha}$$
for every congruence $\alpha$. Given $N\in \N'(Q)$, the relation $\c{N}$ does not need to be a congruence (this is the case for {\it semimedial} and {\it spelling} left quasigroups including racks, see \cite[Theorem 3.5, Theorem 5.6]{LTT}).

The set of admissible subgroups is a sublattice of the lattice of normal subgroups of the left multiplication group (see \cite[Lemma 1.7]{LTT}), and is stable under the usual correspondence between normal subgroups. 
\begin{corollary}\label{iso of lattice 2}\cite[Proposition 2.4]{nilpotentleft}
Let $(Q,\cdot,\backslash)$ be a left quasigroup and $\alpha$ be a congruence of $Q$. Then the mappings 
\begin{align*}
\setof{N\in \N'(Q)}{\dis^\alpha\leq N} &\longleftrightarrow \N'(Q/\alpha)\\
N&\mapsto \pi_\alpha(N)\\
\pi_\alpha^{-1}(K)&\leftmapsto K
\end{align*}
provides an isomorphism of lattices.
\end{corollary}

The following lemma shows the interplay between the maps $\pi_{\alpha}$ as defined in \eqref{pi_map} and the operators $\mathcal{O}_*$ and $\c{*}$.
\begin{lemma}\label{image of orbit congruences}
Let $(Q,\cdot, \backslash)$ be a left quasigroup, $\alpha\in Con(Q)$, $\beta=\mathcal{O}_{\dis^\alpha}$ and $N\in \N'(Q)$. If $\dis^\alpha\leq N$ then:
\begin{itemize}
\item[(i)] $\mathcal{O}_N/\beta=\mathcal{O}_{\pi_\beta(N)}$.
\item[(ii)]  $\c{N}/\alpha=\c{\pi_\alpha(N)}$.
\end{itemize}
\end{lemma}

\begin{proof}
Since $\dis^\alpha\leq N$ we have that $\beta\leq \mathcal{O}_N$ and $\alpha\leq \c{\dis^\alpha}\leq \c{N}$.

(i) By definition $[x]_\beta \, \mathcal{O}_{N}/\beta \, [y]_\beta$ if and only if $x=h(y)$ for some $h\in N$. Thus, if $x=h(y)$ it follows that $[x]_\beta=[h(y)]_\beta=\pi_\beta(h)([y]_\beta)$ and so $[x]_\beta \,\mathcal{O}_{N}/\beta \, [y]_\beta$. On the other hand if $[x]_\beta=\pi_\beta(h)([y]_\beta)=[h(y)]_\beta$ then $x\, \beta\, h(y)$, i.e. $x=kh(y)$ for some $k\in \dis^\alpha$. Since $\dis^\alpha\leq N$ we have that $hk\in N$ and thus $x \, \mathcal{O}_N y$.
%
%
%

(ii) If $x\, \c{N}\, y$ then $\pi_\alpha(L_x L_y^{-1})=L_{[x]_\alpha} L_{[y]_\alpha}^{-1}\in \pi_\alpha(N)$. On the other hand, if $[x]_\alpha\, \c{\pi_{\alpha}(N)}\, [y]_\alpha$ then $L_x L_y^{-1}\in N\dis^\alpha=N$.
\end{proof}
 
%
%
%
%
%

\subsection{The Cayley kernel}
%

%

The {\it Cayley kernel} of a left quasigroup $Q$ is the equivalence relation $\lambda_Q=\c{1}$, i.e. given $x,y\in Q$ then
$$x\, \lambda_Q\, y\, \text{ if and only if } L_x=L_y.$$
In general, the Cayley kernel is not a congruence. If this is the case, we say that $Q$ is a {\it Cayley left quasigroup} (e.g. racks are Cayley left quasigroups).

\begin{remark}\label{remark on lambda}\cite[Remark 2.7]{nilpotentleft} 
Let $(Q,\cdot, \backslash)$ be a left quasigroup and $\alpha\in Con(Q)$:
\begin{itemize}
\item[(i)] $\alpha\leq \lambda_Q$ if and only if $\dis_\alpha=1$;
\item[(ii)] $\lambda_{Q/\alpha}=\c{\dis^\alpha}/\alpha$. 
\end{itemize}
\end{remark}
In \cite{covering_paper} we show that {\it strongly abelian congruences} of left quasigroups in the sense of \cite{TCT} are those below the Cayley kernel. In particular, congruence arising from orbits are related to strongly abelian congruences.

%
%
%
%
%

\begin{lemma}\label{lemma below lambda}\cite[Lemma 2.9]{nilpotentleft}
Let $(Q,\cdot, \backslash)$ be a left quasigroup, $\alpha\in Con(Q)$, $\beta=\mathcal{O}_{\dis^\alpha}$ and $\gamma=\mathcal{O}_{\dis_\alpha}$. Then $\alpha/\beta\leq \lambda_{Q/\beta}$ and $\alpha/\gamma\leq \lambda_{Q/\gamma}$.
\end{lemma}

%

%

A left quasigroup $Q$ is said to be {\it \what} if given $\alpha,\beta\in Con(Q)$, $\alpha\leq\beta$ and $\beta/\alpha\leq \lambda_{Q/\alpha}$ then $\alpha=\beta$. The property of being \what has also an universal algebraic interpretation: indeed a left quasigroup $Q$ is \what if and only if every homomorphic image of $Q$ omit strongly abelian congruences.

\begin{corollary}\label{staB_under_H}
 The class of \what left quasigroups is closed under homorphic images.
\end{corollary}

Let us show a characterization of \what left quasigroups in terms of the operators $\mathcal{O}_*$, $\dis_*$ and $\dis^*$.

\begin{lemma}\label{sharp}
Let $(Q,\cdot, \backslash)$ be a left quasigroup. The following are equivalent:
\begin{itemize}
\item[(i)] $Q$ is \what. 
\item[(ii)] $\mathcal{O}_{\dis_\alpha}=\mathcal{O}_{\dis^\alpha}=\alpha$ for every $\alpha\in Con(Q)$.
\end{itemize}
\end{lemma}

\begin{proof}
(i) $\Rightarrow$ (ii) According to Lemma \ref{lemma below lambda} We have $\gamma=\mathcal{O}_{\dis_\alpha}\leq\beta= \mathcal{O}_{\dis^\alpha}\leq \alpha$ and $\alpha/\gamma\leq \lambda_{Q/\gamma}$ for every $\alpha\in Con(Q)$. Hence, if $Q$ is \what $\gamma=\beta=\alpha$.

(ii) $\Rightarrow$ (i) Let $\alpha,\beta\in Con(Q)$.	Assume that $\beta\leq \alpha$ and $\alpha/\beta\leq \lambda_{Q/\beta}$. Then $\dis_{\alpha/\beta}=1$, i.e. $\dis_\alpha\leq \dis^\beta$. So we have that
$$\alpha=\mathcal{O}_{\dis_\alpha}\leq \mathcal{O}_{\dis^\beta}=\beta$$
and so $\alpha=\beta$.
\end{proof}

It is immediate to see that if $Q$ is a \what left quasigroup, then $\mathcal{O}_*$ is onto and $\dis^*$ and $\dis_*$ are injective on $Q$.

A left quasigroup $Q$ is {\it faithful} if $\lambda_Q=0_Q$ (note that faithful left quasigroups are Cayley). 


%


\begin{lemma}\label{lemma O=con}
Let $(Q,\cdot, \backslash)$ be a left quasigroup. The following are equivalent:

\begin{itemize}
\item[(i)] $\mathcal{O}_*=\c{*}$ on $Q$.
\item[(ii)] $\mathcal{O}_{\dis_\alpha}=\mathcal{O}_{\dis^\alpha}=\alpha=\c{\dis_\alpha}=\c{\dis^\alpha}$ for every $\alpha\in Con(Q)$.
\item[(iii)] $Q/\alpha$ is faithful for every $\alpha\in Con(Q)$.
\end{itemize}
In particular, if (i) holds then $Q$ is \what.
\end{lemma}

\begin{proof}
(i) $\Rightarrow$ (ii) Let $\alpha\in Con(Q)$.	Since we have
$$\mathcal{O}_{\dis_\alpha}\leq\mathcal{O}_{\dis^\alpha}\leq \alpha\leq \c{\dis_\alpha}\leq \c{\dis^\alpha}$$
if $\mathcal{O}_*=\c{*}$ equality holds.

(ii) $\Rightarrow$ (iii) According to Remark \ref{remark on lambda}(ii) $\lambda_{Q/\alpha}=\c{\dis^\alpha}/\alpha$. So we have that $\lambda_{Q/\alpha}=0_Q$.

(iii) $\Rightarrow$ (i) Let $N\in \N(Q)$ and let $\alpha=\mathcal{O}_N\leq \c{N}$. Thus, $\lambda_{Q/\alpha}=\c{\dis^\alpha}/\alpha=0_{Q/\alpha}$ i.e. $\c{\dis^\alpha}=\alpha$. Since $N\leq \dis^\alpha$ then $\c{N}\leq \c{\dis^\alpha}=\alpha$, and so equality holds.
\end{proof}

The property in Lemma \ref{lemma O=con}(iii) is clearly closed under homomorphic images. 

%
%
%
%
%

\section{Galois connections for left quasigroups}\label{Galois}
Let $(A, \leq)$ and $(B, \leq)$ be two posets. A {\it monotone Galois connection} between these posets consists of two monotone functions: $F : A \longrightarrow B$ and $G : B \longrightarrow A$, such that for all $x \in A$ and $y \in B$, we have
$$F(x) \leq y\, \text{ if and only if } \, x \leq G(y).$$

If $(F,G)$ is a Galois connection between two posets, then $FGF=F$ and $GFG=G$. In particular, if $F$ (resp. $G$) is an isomorphism then $F^{-1}=G$. In such case we say that the pair $(F,G)$ provides an isomorphism between the two posets.



The following Galois connection for left quasigroups was first introduced for racks in \cite{CP} and then for left quasigroups in general in \cite{LTT}.

\begin{theorem}\label{galois_connection}\cite[Theorem 1.10]{LTT}
Let $(Q,\cdot, \backslash)$ be a left quasigroup. The pair of mappings $\mathcal{O}_*$ and $\dis^{*}$ provides a monotone Galois connection between $Con(Q)$ and $\N'(Q)$.
\end{theorem}

Let $(Q,\cdot, \backslash)$ be a left quasigroup and $\alpha\in Con(Q)$. According to Lemma \ref{pi}, Corollary \ref{iso of lattice 2} and Theorem \ref{galois_connection} the following diagram (where the vertical arrow labeled by $*/\alpha$ is the canonical isomorphism $\beta\mapsto \beta/\alpha$) is commutative:
\begin{equation}\label{comuting diagram}
\xymatrixcolsep{6pc}\xymatrix{
\setof{\beta\in Con(Q)}{\alpha\leq \beta}  \ar[r]^{\dis^*}\ar[d]^{*/\alpha} & \setof{N\in \N'(Q)}{\dis^\alpha\leq N}\ar[d]^{\pi_\alpha} \\
Con(Q/\alpha) \ar[r]^{\dis^*}& \N'(Q/\alpha)
}
\end{equation}

 Let $\beta=\mathcal{O}_{\dis^\alpha}$. The pair $(\dis^*,\mathcal{O}_*)$ provides a Galois connection and so $\dis^\alpha=\dis^{\beta}$. According to Lemma \ref{image of orbit congruences}(i) we have the following commuting diagram:
\begin{equation}
\xymatrixcolsep{6pc}\xymatrix{ \setof{N\in \N'(Q)}{\dis^\beta\leq N}\ar[r]^{\mathcal{O}_*}\ar[d]^{\pi_\beta} & 
\setof{\gamma\in Con(Q)}{\beta\leq \gamma}  \ar[d]^{*/\beta} \\
\N'(Q/\beta) \ar[r]^{\mathcal{O}_*}& Con(Q/\beta)
}\label{comuting diagram_2}
\end{equation}

We say that $Q$ has the {\it congruence determined by orbits} (shortly, CDOs) if the pair $(\dis^*,\mathcal{O}_*)$ provides an isomophism between $Con(Q)$ and $\N(Q)$. Note that if $Q$ has CDOs then $1_Q=\mathcal{O}_{\dis(Q)}$, namely $Q$ is connected by $\dis(Q)$.

By Corollary \ref{iso of lattice 2}, the map $\pi_\alpha$ is bijective between $\setof{N\in \N'(Q)}{\dis^\alpha\leq N}$ and $\N'(Q/\alpha)$. Then according to diagram \eqref{comuting diagram}, if $\dis^*$ is injective (resp. onto) on $Q$ then $\dis^*$ is also injective (resp. onto) on $Q/\alpha$.


\begin{corollary}\label{factor of FG 1}
The class of left quasigroups having CDOs is closed under homomorphic images.
\end{corollary}

The following example shows that having CDOs is a relevant property.

\begin{example}
Let $Q$ be a simple left quasigroup, $N\in \N'(Q)$ and $\alpha=\mathcal{O}_N$. So either $\alpha=0_Q$ and so $N=1$, or $\alpha=1_Q$ and thus $\dis(Q)=\dis_\alpha\leq N$. Therefore, $\N'(Q)=\{1, \dis(Q)\}$. Note that $\dis(Q)=1$ if and only if $\lambda_Q=1_Q$. We have to discuss two cases:
\begin{itemize}
\item[(i)] Let $\lambda_Q=1_Q$. If $Q$ is idempotent then $Q\cong\mathcal{P}_2$, otherwise $Q\cong \aff(\Z_p,0,1,1)$ for $p$ prime.  

\item[(ii)] If $\lambda_Q\neq 1_Q$, we have $\N'(Q)=\{1, \dis(Q)\}$ and accordingly $Q$ has CDOs.
\end{itemize}
In particular, simple idempotent left quasigroups with size bigger than $2$ have CDOs.
 \end{example}

%
%



Let us introduce a second Galois connection for left quasigroups.

\begin{proposition}\label{galois_connection_2}
Let $(Q,\cdot, \backslash)$ be a left quasigroup. The pair of mappings $\c{*}$ and $\dis_{*}$ provides a monotone Galois connection between $Equiv(Q)$ and the lattice of normal subgroups of $\lmlt(Q)$. 
\end{proposition}

\begin{proof}
If $\alpha\leq \c{N}$, then $\setof{L_x L_y^{-1}}{x,\,\alpha\, y}\subseteq N$ and therefore $\dis_\alpha\leq N$. On the other hand, if $\dis_\alpha\leq N$, then $\alpha\leq \c{\dis_\alpha}\leq \c{N}$.
\end{proof}

Let $(Q,\cdot, \backslash)$ be a left quasigroup and $\alpha\in Con(Q)$. According to Lemma \ref{pi}, Corollary \ref{iso of lattice 2} and Theorem \ref{galois_connection} the following diagram is commutative:

\begin{equation}
\label{comuting diagram_1}
\xymatrixcolsep{6pc}\xymatrix{
\setof{\beta\in Con(Q)}{\alpha\leq \beta}  \ar[r]^{\dis_*}\ar[d]^{*/\alpha} & \setof{N\in \N'(Q)}{\dis_\alpha\leq N}\ar[d]^{\pi_\alpha} \\
Con(Q/\alpha) \ar[r]^{\dis_*}& \N'(Q/\alpha)
}
\end{equation}

Let $(Q,\cdot, \backslash)$ be a left quasigroup and $\alpha\in Con(Q)$. The following diagram is commutative by Lemma \ref{image of orbit congruences}(ii):
\begin{equation}
\xymatrixcolsep{6pc}\xymatrix{ \setof{N\in \N'(Q)}{\dis^\alpha\leq N}\ar[r]^{\c{*}}\ar[d]^{\pi_\alpha} & 
\setof{\beta\in Equiv(Q)}{\alpha\leq \beta}  \ar[d]^{*/\alpha} \\
\N'(Q/\alpha) \ar[r]^{\c{*}}& Equiv(Q/\alpha)
}\label{comuting diagram_3}
\end{equation}

We say that $Q$ has the {\it congruence determined by subgroups} (shortly, CDSg) if the pair $(\dis_*,\c{*})$ provides an isomophism between $Con(Q)$ and $\N'(Q)$.
%

%
%

The map $\pi_\alpha$ is bijective between $\setof{N\in \N'(Q)}{\dis^\alpha\leq N}$ and $\N'(Q/\alpha)$. Then according to diagram \eqref{comuting diagram_3}, if $\c{*}$ is injective (resp. onto) on $Q$ then $\c{*}$ is also injective (resp. onto) on $Q/\alpha$.

\begin{corollary}\label{factor of FG 2}
The class of left quasigroups having CDSg is closed under homomorphic images.
\end{corollary}
%
%
%
%
%
%
%
%
%
%
%

 \begin{example}\label{example of CDSg} 
 
 Let us show some examples of left quasigroups having CDSg.
\begin{itemize}

\item[(i)] Let $Q$ be a simple Cayley left quasigroup. If $\lambda_Q=1_Q$ then either $Q=\mathcal{P}_2$ or $Q=\aff(\Z_p,0,1,1)$ for $p$ prime. Otherwise $\lambda_Q=0_Q$ and $Q$ has CDSg.
%
%
\item[(ii)] Superconnected idempotent semiregular left quasigroup are quandles (see \cite[Lemma 3.3]{nilpotentleft}). Such quandles have CDSg \cite[Proposition 2.12]{Super}. 
\end{itemize}

\end{example}


%

%

The following proposition extends \cite[Proposition 3.13]{CP}. 

%
%

\begin{theorem}\label{CDSg vs CDOs}
Let $(Q,\cdot, \backslash)$ be a left quasigroup. The following are equivalent:
\begin{itemize}
\item[(i)]  $Q$ has CDSg.

\item[(ii)] $Q$ is \what and has CDOs.


\item[(iii)] $\c{*}=\mathcal{O}_*$ and $\dis_*=\dis^*$ on $Q$.

\end{itemize}
\end{theorem}


\begin{proof}
(i) $\Rightarrow$ (ii) Clearly $\dis_{\lambda_Q}=\dis_{\c{1}}=1=\dis_{0_Q}$ and so $\lambda_Q=0_Q$. Therefore $Q$ and its factors are all faithful. By Lemma \ref{lemma O=con} $Q$ is \what and we have that $\c{*}=\mathcal{O}_*$. Since $\c{*}$ is an isomorphism between $\N'(Q)$ and $Con(Q)$ then so it is $\mathcal{O}_*$.

(ii) $\Rightarrow$ (iii) By Lemma \ref{sharp} we have that $\mathcal{O}_{\dis_\alpha}=\mathcal{O}_{\dis^\alpha}$ for every $\alpha\in Con(Q)$. Since $Q$ has CDOs we have that $\dis_*=\dis^*$ on $Q$ and accordingly $\mathcal{O}_*=(\dis^*)^{-1}=(\dis_*)^{-1}=\c{*}$ on $Q$.

(iii) $\Rightarrow$ (i) Let $N\in \N'(Q)$ and $\alpha=\mathcal{O}_{N}=\c{N}$. Then $\dis_{\alpha}\leq N\leq \dis^\alpha$ and so $N=\dis_{\c{N}}$. On the other hand, according to Lemma \ref{lemma O=con} we have that $\alpha=\c{\dis_\alpha}$ for every $\alpha\in Con(Q)$.
\end{proof}

Some classes of left quasigroups are know to be \what:

\begin{itemize}
\item[(i)] Malt'sev left quasigroups (indeed they omit strongly abelian congruences \cite{TCT}).
\item[(ii)] Superconnected idempotent left quasigroups (indeed they are faithful and so they are all their factors \cite{Super}). 
\end{itemize}



%
%

\begin{corollary}\label{CDSg3}
Let $Q$ be a Malt'sev (resp. idempotent superconnected) left quasigroup. The following are equivalent:
\begin{itemize}
\item[(i)] $Q$ has CDOs.
\item[(ii)] $\Q$ has CDSg.
\end{itemize}
%
\end{corollary}

The following example shows that having CDOs and having CDSg are dinstinct properties in general.
\begin{example}\label{example 8,1}
The quandle $Q=${\tt SmallQuandle(8,1)} in the RIG database of GAP \cite{GAP4} has CDOs, but it has not CDSg (indeed $Q$ is not \what).
\end{example}

\section{Nilpotency and CDSg property}

\subsection{Commutator theory and central extensions}

We recall the basics of commutator theory for the reader's convenience. The commutator theory for arbitrary algebraic structures have been developed in \cite{Smith, comm}, including the definition of commutator of congruences and the related notions of center, solvability and nilpotency. We will present all the definitions just for left quasigroups.

Let $(Q,\cdot, \backslash)$ be a left quasigroup and $\alpha,\beta,\delta\in Con(Q)$. We say that \emph{$\alpha$ centralizes $\beta$ over $\delta$}, and write $C(\alpha,\beta;\delta)$, if for every $(n+1)$-ary term operation $t$, every pair $x\,\alpha\,y$ and every $z_1\,\beta\,u_1$, $\dots$, $z_n\,\beta\,u_n$ we have
\[  t^Q(x,z_1,\dots,z_n)\,\delta \, t^Q(x,u_1,\dots,u_n)\quad\text{implies}\quad t^Q(y,z_1,\dots,z_n) \, \delta \, t^Q(y,u_1,\dots,u_n). \]

We denote by $[\alpha,\beta]$ the \emph{commutator} of $\alpha,\beta\in Con(Q)$, that is the smallest congruence $\delta$ such that $C(\alpha,\beta;\delta)$. 
A congruence $\alpha$ is called:
\begin{itemize}
	\item \emph{abelian} if $C(\alpha,\alpha;0_Q)$, i.e., if $[\alpha,\alpha]=0_Q$,
	\item \emph{central} if $C(\alpha,1_Q;0_Q)$, i.e., if $[\alpha,1_Q]=0_Q$.
\end{itemize}
The largest congruence of $Q$ that centralizes $1_Q$ is called the \emph{center} of $Q$ and denoted by $\zeta_Q$. A left quasigroup $Q$ is called \emph{abelian} if $\zeta_Q=1_Q$, or, equivalently, if the congruence $1_Q$ is abelian. We can define a series of congruence of $Q$ that plays the same role as the series of centers for groups. The definition is the following:
$$\zeta_1(Q)=\zeta_Q,\qquad \zeta_{n+1}(Q)/\zeta_{n}(Q)=\zeta_{(Q/\zeta_{n}(Q))}$$
for every $n\in \mathbb{N}$. The left quasigroup $Q$ is called \emph{nilpotent} of length $n$ if $\zeta_n(Q)=1_Q$.
%
%

The theory of commutators have been adapted to the setting of racks and quandles in \cite{CP}. The main results can also be partially applied to the setting of left quasigroups.

\begin{lemma}\cite[Lemma 5.1]{CP}\label{from CP}
Let $(Q,\cdot, \backslash)$ be a left quasigroup, $\alpha,\beta\in Con(Q)$. If $C(\alpha,\beta;0_Q)$ holds then $[\dis_{\alpha},\dis_{\beta}]=1$ and $\left(\dis_\beta\right)_x=\left(\dis_\beta\right)_y$  whenever $x\, \alpha\, y$.
\end{lemma}

\begin{corollary}\label{corollary from CP}
Let $(Q,\cdot, \backslash)$ be a left quasigroup and $\alpha\leq \zeta_Q$. Then $\dis_\alpha$ is central in $\dis(Q)$ and $\dis(Q)_x=\dis(Q)_y$ whenever $x\, \alpha\, y$.
\end{corollary}

%
%

Let us present a standard contruction of left quasigroups involving abelian groups. Later in the section we are proving that such construction is related to central congruences. Let $(Q,\cdot, \backslash)$ be a left quasigroup, $A$ an abelian group, $f\in \aut{A}$, $g\in \End{A}$ and $\theta:Q\times Q\longrightarrow A$ be a map. We can define the left quasigroup $E=(Q\times A, \cdot,\backslash)$ where
\begin{equation}\label{central ext}
(x,a)\cdot(y,b)=(x\cdot y, g(a)+f(b)+\theta(x,y)).
\end{equation}
We denote such a structure by $\aff(Q,A,g,f,\theta)$ and we say that $E$ is a {\it central extension} of $Q$ by $A$.  

Moreover, the map
$$p_1:E\longrightarrow Q,\quad (x,a)\mapsto x$$
is a morphism of left quasigroups. 

If $|Q|=1$, we can identify $Q\times A$ with $A$ and \eqref{central ext} reads
$$a\cdot b=g(a)+f(b)+c$$
for some $c\in A$. For this special case we use the notation $E=\aff(A,g,f,c)$ and we say that $E$ is an {\it affine left quasigroup} over $A$.

\begin{lemma}\label{on terms of central ext}
Let $(Q,\cdot, \backslash)$ be a left quasigroup, $A$ an abelian group and $E=\aff(Q,A,g,f,\theta)$. Then every term operations of $E$ has the form $$t^E((x_1,a_1),\ldots,(x_n,a_n))=(t^Q(x_1,\ldots x_n),\sum_{j=1}^n G_j(a_j)+\Theta(x_1,\ldots x_n))$$
where $G_j\in \End{A}$ for every $j=1,	\ldots n$ and $\Theta:Q^{n}\longrightarrow A$.
\end{lemma}

\begin{proof}
Let us prove the statement by induction on the number of occurrences of variables. If the term $t$ is just a variable there is nothing to prove. Let us denote $\overline{x}=(x_1,\ldots,x_n)$ and $\overline{(x,a)}=((x_1,a_1),\ldots, (x_n,a_n))$. Then
$$t^E(\overline{(x,a)})=L_{s_1^E(\overline{(x,a)})} s_2^E(\overline{(x,a)})$$
for suitable subterms $s_1$ and $s_2$. Then by induction we have that
$$s_i^E(\overline{(x,a)})=(s_i^Q(\overline{x}), \sum_{j} G_{i,j}(a_j)+\Theta_{i}(\overline{x}))$$ for $i=1,2$. Therefore
\begin{align*}
 L_{s_1^E(\overline{(x,a)})} s_2^E(\overline{(x,a)})&= (L_{s_1^Q(\overline{x})} s_2^Q(\overline{x})), g\left(\sum_j G_{1,j}(a_j)+\Theta_{1}(\overline{x})\right)+f\left(\sum_j G_{2,j}(a_j)+\Theta_{2}(\overline{x})\right)+\theta(s_1^Q(\overline{x}),s_2^Q(\overline{x})))\\
&=(t^Q(\overline{x}), \sum_j \underbrace{(gG_{1,j}+f G_{2,j})}_{\in \End{A}}(a_j)	+\underbrace{g(\Theta_{1}(\overline{x}))+f(\Theta_{2}(\overline{x}))+\theta(s_1^Q(\overline{x}),s_2^Q(\overline{x}))}_{=\Theta'(\overline{x})})
\end{align*}
and so the statement follows.
\end{proof}

\begin{corollary}\label{ker_p_1}
Let $(Q,\cdot, \backslash)$ be a left quasigroup and $E=\aff(Q,A,g,f,\theta)$. Then $\ker{p_1}$ is a central congruence of $E$.
\end{corollary}

\begin{proof}
According to Lemma \ref{on terms of central ext}, we have
$$t^E((x,a),(y_1,b_1)\ldots,(y_{n-1},b_{n-1}))=(t^Q(x,y_1,\ldots, y_{n-1}),G(a)+\sum_j G_j(b_j)+\Theta(x,y_1,\ldots y_{n-1}))$$
where $G,G_j\in \End{A}$ and $\Theta:Q^{n}\longrightarrow A$. If the equality $$t^E((x,a),(y_1,b_1)\ldots,(y_{n-1},b_{n-1}))=t^E((x,a),(z_1,c_1)\ldots,(z_{n-1},c_{n-1}))$$ holds then 
\begin{align*}
t^Q(x,y_1,\ldots, y_{n-1})&=t^Q(x,z_1,\ldots, z_{n-1}),\\
\sum_j G_j(b_j)+\Theta(x,y_1,\ldots,y_{n-1})&=\sum_j G_j(c_j)+\Theta(x,z_1,	\ldots,z_{n-1}).
\end{align*}

Hence we also have 
\begin{align*}
t^E((x,d),(y_1,b_1)\ldots (y_{n-1},b_{n-1}))&=(t^Q(x,y_1,\ldots, y_{n-1}),G(d)+\sum_j G_j(b_j)+\Theta(x,y_1,\ldots y_{n-1}))\\
&=(t^Q(x,z_1,\ldots, z_{n-1}),G(d)+\sum_j G_j(c_j)+\Theta(x,z_1,\ldots z_{n-1}))\\&=t^E((x,d),(z_1,c_1),\ldots (z_{n-1},c_{n-1}))
\end{align*} for every $d\in A$.
\end{proof}

According to Corollary \ref{corollary from CP} $\dis_{\ker{p_1}}\leq Z(\dis(E))$ since $\ker{p_1}$ is a central congruence by Corollary \ref{ker_p_1}. Thus, if $h=w L_{(z,d)}^k$ for $w\in \dis(Q)$ and $k\in \mathbb{Z}$ we have that the action of the generators of $\dis_{\ker{p_1}}$ is given by
\begin{align}\label{dis for central ext}
h L_{(x,a)}L_{(x,b)}^{-1}h^{-1}(y,c)=L_{(z,d)}^k L_{(x,a)}L_{(x,b)}^{-1}L_{(z,d)}^{-k}(y,c)=(y,c+f^kg(a-b))
\end{align}
for every $x,y\in Q$ and every $a,b,c\in A$.

Let $N\leq A$. We define the relation $\alpha_N$ by setting 
\begin{equation}\label{alpha_n}
(x,a)\, \alpha_N\, (y,b)\, \text{ if and only if } \, x=y\, \text{ and }\, a-b\in  N
\end{equation}
for every $(x,a),(y,b)\in Q\times A$.
%

\begin{lemma}\label{lemma quandle central ext}
Let $(Q,\cdot, \backslash)$ be a left quasigroup, $A$ an abelian group, $E=\aff(Q,A,g,f,\theta)$ and $N\leq A$. If $g(N)\leq N=f(N)$, the relation $\alpha_N$
 is a congruence of $E$.
\end{lemma}
\begin{proof}
Let $a-b\in N$. Then, $f^{k}g(a-b)\in N$ for every $k\in \mathbb{Z}$  and so according to \eqref{dis for central ext} we have
$$h L_{(x,a)} L_{(x,b)}^{-1} h^{-1}(y,c)=(y,c+f^{k}g(a-b))\, \alpha_N\, (y,c)$$
for every $(y,c)\in E$. Therefore $\dis_{\alpha_N}\leq \dis^{\alpha_N}$. Moreover, using \eqref{central ext}
\begin{align*}
(y,c)(x,a)&=(yx,g(c)+f(a)+\theta_{y,x})\,\alpha_N\, (y,c)(x,b)=(yx,g(c)+f(b)+\theta_{y,x}),\\
(y,c)\backslash (x,a)&=(y\backslash x,f^{-1}(a-g(c)-\theta_{y,y\backslash x}))\,\alpha_N\,(y,c)\backslash (x,b)=(y\backslash x,f^{-1}(b-g(c)-\theta_{y,y\backslash x})).
\end{align*}%
Therefore $\alpha_N$ provides a system of blocks for the action of $\lmlt(E)$ and so $\alpha_N$ is a congruence according to \cite[Lemma 1.5]{LTT}.
%
\end{proof}
Let us show that the converse of \cite[Lemma 2.3]{nilpotentleft} holds for congruences arising from central extensions.

\begin{lemma}\label{lemma quandle central ext_0}
Let $Q$ be an idempotent left quasigroup, $E=\aff(Q,A,g,f,\theta)$ and $\alpha=\ker{p_1}$. If $\alpha=\mathcal{O}_{\dis_\alpha}$ then the blocks of $\alpha$ are connected.
\end{lemma}

\begin{proof}
The blocks of $\alpha$ are subalgebras of $E$, since $Q$ is idempotent. The group $\dis_\alpha$ is generated by $\setof{hL_{(x,a)}L_{(x,b)}^{-1}h^{-1}}{x\in Q,\, a,b\in A}$.

According to \eqref{dis for central ext} the orbits of $(y,c)$ with respect to the action of $\dis_\alpha$ is $(y,c+H)$ where $H$ is the subgroup generated by $\setof{f^kg(A)}{k\in \mathbb{Z}}$ and it coincides with the orbit with respect to the action of $\dis([x]_\alpha)$ (indeed it is enough to set $x=y=z$ in \eqref{dis for central ext}).
\end{proof}

\subsection{Nilpotent left quasigroups having CDSg}

Note that $E=\aff(Q,A,g,f,\theta)$ is idempotent if and only if $Q$ is idempotent, $g=1-f$ and $\theta_{x,x}=0$. Note that in this case the blocks of $\ker{p_1}$ are affine quandles isomorphic to $\aff(A,1-f,f,0)$.

In this subsection we strengthen \cite[Proposition 3.17]{Principal}, since we extend it to idempotent left quasigroups of arbitrary cardinality. 

\begin{lemma}\label{CDSg for affine}
Let $E=\aff(Q,A,1-f,f,\theta)$ be an idempotent left quasigroup. If $E$ has CDOs then $\aff(A,1-f,f,0)$ is superconnected. 
\end{lemma}

\begin{proof}
Let $a\in A$. The mapping $t_a:x\mapsto x+a$ is an automorphism of $Q'=\aff(A,1-f,f,0)$. So if $M$ is a subquandle of $Q'$, then $M=a+M'\cong M'$ for some subquandle $M'$ containing $0$ and some $a\in A$. So, it is enough to show that the subquandles containing $0$ of $Q'$ are connected. Let $M$ be a subquandle of $Q'$ containing $0$. Let us denote by $N=\langle m, m\in M\rangle\leq A$. Note that $f(m)=0*m\in M$ and $f^{-1}(m)=0\backslash m\in M$ for every $m\in M$. So $(1-f)(N)\leq N=f(N)$. In particular $N$ is a subquandle of $Q'$ and it contains $M$.

According to Lemma \ref{lemma quandle central ext}(ii), the relation $(x,a)\,\alpha_N\, (y,b)$ if and only if $x=y$ and $a-b\in N$ is a congruence of $E$. The quandle $E$ has CDOs, i.e. $\alpha_N=\mathcal{O}_{\dis_{\alpha_N}}$ and by Lemma \ref{lemma quandle central ext_0} the blocks of $\alpha_N$ are connected. In particular $N=[(x,0)]_{\alpha_N}$ is connected. Moreover 
$$\dis(N)=(1-f)(N)=(1-f)(\langle m,\, m\in N\rangle)=\langle(1-f)(m), \, m\in M\rangle=\dis(M).$$
Then we have
$$M\subseteq N=(x,0)^{\dis(N)}=(x,0)^{\dis(M)}\subseteq M,$$
and so $N=M$ and $M$ is connected.
\end{proof}

In particular, Lemma \ref{CDSg for affine} implies that affine quandles having CDOs are superconnected. In the following theorem we characterize nilpotent idempotent left quasigroups having CDSg.

\begin{theorem}\label{nilpotent CDSg_0}
Let $Q$ be a nilpotent idempotent left quasigroup.  The following are equivalent:
\begin{itemize}

\item[(i)] $Q$ has CDSg.
%

\item[(ii)] $Q$ is a semiregular superconnected latin quandle. 
\end{itemize}
\end{theorem}

\begin{proof}

(i) $\Rightarrow$ (ii) By Proposition \ref{CDSg vs CDOs} $Q$ has CDOs and $\dis_*=\dis^*$ on $Q$. If $Q$ is abelian, then it is semiregular and so it is a superconnected quandle by \cite[Lemma 3.3]{nilpotentleft} and Lemma \ref{CDSg for affine}. Assume that $Q$ is nilpotent of length $n+1$. The factor $Q/\zeta_Q$ has CDSg according to Lemma \ref{factor of FG 2}. By induction on the nilpotency length, $Q/\zeta_Q$ is a semiregular superconnected quandle. So we have $\dis(Q)_x\leq \BlocS{Q}{x}=\dis^{\zeta_Q}=\dis_{\zeta_Q}\leq Z(\dis(Q))$ (the first equality follows by Lemma \cite[Lemma 3.1]{nilpotentleft} and the relative displacement group of a central congruence is central by Corollary \ref{corollary from CP}). 
The left quasigroup $Q$ is connected by $\dis(Q)$, therefore $\dis(Q)_x$ is normal and then trivial. So $Q$ is a semiregular quandle by \cite[Lemma 3.3]{nilpotentleft}.
%

In particular $Q$ is connected, faithful and $\zeta_Q=\mathcal{O}_{\dis_{\zeta_Q}}$, so $Q$ is a central extension of $Q/\zeta_Q$ \cite[Proposition 7.9]{CP} and $[x]_{\zeta_Q}\cong \aff(A,1-f,f,0)$. The blocks of $\zeta_Q$ are superconnected by Lemma \ref{CDSg for affine} and accordingly $Q$ is also superconnected (the class of idempotent superconnected left quasigroups is closed under extensions, see \cite[Corollary 1.12]{Super}). 

Finally, superconnected nilpotent quandles are latin \cite[Theorem 2.15]{Super}.
%

(ii) $\Rightarrow$ (i) It follows by Example \ref{example of CDSg}. 
\end{proof}

Note that also for idempotent nilpotent left quasigroups the property of having CDOs and having CDSg are different: indeed the quandle in Example \ref{example 8,1} is nilpotent.

\bibliographystyle{amsalpha}
\bibliography{references} 

\def\cprime{$'$} \def\cprime{$'$}
\providecommand{\bysame}{\leavevmode\hbox to3em{\hrulefill}\thinspace}
\providecommand{\MR}{\relax\ifhmode\unskip\space\fi MR }
\providecommand{\MRhref}[2]{%
  \href{http://www.ams.org/mathscinet-getitem?mr=#1}{#2}
}
\providecommand{\href}[2]{#2}
\begin{thebibliography}{{Bon}22a}

\bibitem[AG03]{AG}
Nicol{\'a}s Andruskiewitsch and Mat{\'{\i}}as Gra{\~n}a, \emph{From racks to
  pointed {H}opf algebras}, Adv. Math. \textbf{178} (2003), no.~2, 177--243.
  \MR{1994219 (2004i:16046)}

\bibitem[Ber12]{UA}
Clifford Bergman, \emph{Universal algebra}, Pure and Applied Mathematics (Boca
  Raton), vol. 301, CRC Press, Boca Raton, FL, 2012, Fundamentals and selected
  topics. \MR{2839398}

\bibitem[BF21]{Maltsev_paper}
Marco {Bonatto} and Stefano {Fioravanti}, \emph{{Mal'cev classes of
  left-quasigroups and Quandles}}, Quasigroups and related systems \textbf{29}
  (2021), no.~2.

\bibitem[Bon20]{Principal}
Marco Bonatto, \emph{Principal and doubly homogeneous quandles}, Monatshefte
  f{\"u}r Mathematik \textbf{191} (2020), no.~4, 691--717.

\bibitem[{Bon}21]{LTT}
Marco {Bonatto}, \emph{Medial and semimedial left quasigroups}, Journal of
  Algebra and its applications \textbf{21} (2021), no.~2.

\bibitem[{Bon}22a]{nilpotentleft}
Marco {Bonatto}, \emph{{Nilpotent left quasigroups}}, arXiv e-prints (2022),
  arXiv:2103.06604.

\bibitem[Bon22b]{Super}
Marco Bonatto, \emph{Superconnected left quasigroups and involutory quandles},
  Communications in Algebra \textbf{50} (2022), no.~9, 3978--3994.

\bibitem[BS19]{covering_paper}
Marco {Bonatto} and David {Stanovsk{\'y}}, \emph{{A Universal algebraic
  approach to rack coverings}}, arXiv e-prints (2019), arXiv:1910.09317.

\bibitem[BS21]{CP}
\bysame, \emph{{Commutator theory for racks and quandles}}, J. Math. Soc. Japan
  \textbf{73(1)} (2021), 41--75.

\bibitem[FM87]{comm}
Ralph Freese and Ralph McKenzie, \emph{Commutator theory for congruence modular
  varieties}, London Mathematical Society Lecture Note Series, vol. 125,
  Cambridge University Press, Cambridge, 1987. \MR{909290}

\bibitem[H.76]{Smith}
Smith J.~D. H., \emph{Mal'tsev varieties}, Adv. Math. (1976). \MR{1994219
  (2004i:16046)}

\bibitem[HD88]{TCT}
McKenzie~R. Hobby~D., \emph{The structure of finite algebras}, Contemporary
  Mathematics, vol.~76, American Mathematical Society, 1988.

\bibitem[Joy82]{J}
David Joyce, \emph{A classifying invariant of knots, the knot quandle}, J. Pure
  Appl. Algebra \textbf{23} (1982), no.~1, 37--65. \MR{638121 (83m:57007)}

\bibitem[Mat82]{Matveev}
S.~V. Matveev, \emph{Distributive groupoids in knot theory}, Mat. Sb. (N.S.)
  \textbf{119(161)} (1982), no.~1, 78--88, 160. \MR{672410}

\bibitem[Ven15]{GAP4}
Leandro Vendramin, \emph{Rig, a gap package for racks, quandles and nichols
  algebras}, The GAP~Group, 2015.

\end{thebibliography}

\end{document}